\documentclass[twoside,11pt]{article}

\usepackage{graphicx, subfigure}
\usepackage[top=1in,bottom=1in,left=1in,right=1in]{geometry}
\usepackage[sort&compress]{natbib} 
\usepackage{amsfonts, amsmath, amssymb, amsthm, constants, bbm}
\usepackage{MnSymbol}
\usepackage{algorithm, algorithmic}
\usepackage{multirow, booktabs}

\usepackage{color}
\definecolor{darkred}{RGB}{100,0,0}
\definecolor{darkgreen}{RGB}{0,100,0}
\definecolor{darkblue}{RGB}{0,0,150}

\usepackage{hyperref}
\hypersetup{colorlinks=true, linkcolor=darkred, citecolor=darkgreen, urlcolor=darkblue}
\usepackage{url}

\newtheorem{thm}{Theorem}
\newtheorem{prp}{Proposition}

\theoremstyle{remark}

\newtheorem{exa}{Example}

\def\beq{\begin{equation}} 
\def\eeq{\end{equation}}
\def\beqn{\begin{eqnarray*}}
\def\eeqn{\end{eqnarray*}}
\def\Bitem{\begin{itemize}\setlength{\itemsep}{.2in}}
\def\bitem{\begin{itemize}\setlength{\itemsep}{.05in}}
\def\eitem{\end{itemize}}
\def\Benum{\begin{enumerate}\setlength{\itemsep}{.2in}}
\def\benum{\begin{enumerate}\setlength{\itemsep}{.05in}}
\def\eenum{\end{enumerate}}
\def\bmult{\begin{multline*}}
\def\emult{\end{multline*}}
\def\bcenter{\begin{center}}
\def\ecenter{\end{center}}
\def\bframe{\begin{frame}}
\def\eframe{\end{frame}}
\def\balign{\begin{align}}
\def\ealign{\end{align}}

\newcommand{\thmref}[1]{Theorem~\ref{thm:#1}}
\newcommand{\prpref}[1]{Proposition~\ref{prp:#1}}

\newcommand{\secref}[1]{Section~\ref{sec:#1}}

\def\bbI{\mathbb{I}}

\def\bbR{\mathbb{R}}

\renewcommand{\P}{\operatorname{\mathbb{P}}}

\def\eps{\varepsilon}

\def\1{\mathbbm{1}}

\definecolor{purple}{rgb}{0.4,.1,.9}

\newcommand\blfootnote[1]{%
  \begingroup
  \renewcommand\thefootnote{}\footnote{#1}%
  \addtocounter{footnote}{-1}%
  \endgroup
}

\begin{document}
\thispagestyle{empty}

\title{Concentration of Measure for Radial Distributions and Consequences for Statistical Modeling}
\author{Ery Arias-Castro \and Xiao Pu}
\date{}
\maketitle

\blfootnote{Both authors are with the Department of Mathematics, University of California, San Diego, USA.  Contact \href{http://math.ucsd.edu/~eariasca}{Ery Arias-Castro} or \href{http://www.math.ucsd.edu/~xipu/}{Xiao (Victor) Pu}.  We are grateful to Chris Sherlock and Daniel Elton for bringing their paper to our attention, and to Gabor Lugosi for helpful discussions.  This work was partially supported by a grant from the US Office of Naval Research (N00014-13-1-0257).}  

\vspace{-0.3in}

\begin{abstract}
Motivated by problems in high-dimensional statistics such as mixture modeling for classification and clustering, we consider the behavior of radial densities as the dimension increases.  We establish a form of concentration of measure, and even a convergence in distribution, under additional assumptions.  This extends the well-known behavior of the normal distribution (its concentration around the sphere of radius square-root of the dimension) to other radial densities.  We draw some possible consequences for statistical modeling in high-dimensions, including a possible universality property of Gaussian mixtures.
\end{abstract}

\section{Introduction} \label{sec:intro}

Nonparametric density estimation can be quickly difficult in high-dimensions because of the curse of dimensionality.  Additional assumptions are often needed.  The most popular one might well be the Naive Bayes approach, popular in classification \citep{lewis1998naive}, which presumes that the variables are independent, or equivalently, that the density is the product of its marginals.  

Another possibility is to assume that the density is elliptical, a classical assumption in multivariate analysis \citep{MR1990662}, meaning that $f$ is of the form $f(x) = |A| g(\|A x\|)$, where $A$ is a positive definite matrix.
Any centered and non-degenerate normal distribution is elliptical, with base distribution the standard normal distribution and $A = \Sigma^{-1/2}$, where $\Sigma$ is the covariance matrix.  Note that the same distribution is also the product of its marginals when $\Sigma$ is diagonal --- the assumption underlying Linear Discriminant Analysis, an important parametric special case of Naive Bayes.
High-dimensional density estimation is of course crucial in classification and clustering.  

Our motivation comes from such problems, and in particular, mixture modeling in high-dimensions using elliptical distributions.  In the process of working on this problem we uncovered a difficulty which we elaborate upon in the present paper: that of estimating the base function.  
Indeed, it is quite tempting to extend the Gaussian mixture models to models of the form
\beq
\sum_{k=1}^K \pi_k |A_k| g_k(\|A_k x\|),
\eeq
where we assume the mixture has $K$ components, with the $k$th component having weight $\pi_k$ and density $|A_k| g_k(\|A_k x\|)$.  For example, \cite{bickel1998efficient} and more recently \cite{bhattacharyyaadaptive} consider models of this kind.  Instead of smoothness assumptions, we are more interested here in shape assumptions, for example that $g_k$ is decreasing and/or log-concave on $\bbR_+$.  
\cite{chang2007clustering} consider such mixture models but under the Naive Bayes assumption instead of assuming the densities are elliptical.
An EM approach to fitting such a model involves being able to estimate $g_k$ based on a sample from $g_k(\|x\|)$.  And this is what we found challenging in our investigation.  

Focusing on this task, suppose we have an IID sample from $f$, where $f$ rotationally invariant (aka radial), meaning that $f(x) = g(\|x\|)$ for some function $g$, and consider the problem of estimating $g$.  In fact, we can work with the magnitudes (the norms of the observations), which are sufficient.  We explain the difficulty of estimating $g$ by the fact that the magnitudes are highly concentrated as the dimension becomes large.  

The simplest case of this concentration of measure phenomenon arises when we assume that that $g \propto \psi$, where $\psi$ is fixed, as is the case in the Gaussian setting.  Specifically, we assume we are in dimension $d+1$ and we work with $\psi$ satisfying the following assumptions (and some additional assumptions specified later on)
\beq\label{psi}
\text{$\psi : \bbR_+ \to \bbR_+$ is such that $1/c_d := \int_0^\infty u^d \psi(u) {\rm d}u < \infty$ for all $d \ge 1$.}
\eeq
We let $X_d$ denote a random variable with density $c_d \psi(\|x\|)$ on $\bbR^{d+1}$ and let $U_d$ denote its magnitude, $U_d = \|X_d\|$, which has density $c_d u^d \psi(u)$ on $\bbR_+$.  

In this context we show a form of concentration of measure, and convergence in distribution, as the dimension $d$ increases.  Concentration is a well-known phenomenon in high-dimensions, in particular for product distributions (Naive Bayes), with far-reaching consequences \citep{ledoux2005concentration,boucheron2013concentration}. 
For radial distributions, it is not as well-known, except for when the density is Gaussian or uniform on a ball.  (The latter is often used to explain some forms of curse of dimensionality.)  
This case was recently studied in detail by \cite{sherlock2012class}, who cite older work by \cite{diaconis1984asymptotics,diaconis1987dozen} in rather special cases.  We comment on \citep{sherlock2012class} in more detail in \secref{discussion}, after we present our results.

In the remaining of the paper we study what happens when $d \to \infty$.  In \secref{compact} we study the case where $\psi$ has compact support, which is the simplest situation.  In \secref{non-compact} we consider the case where $\psi$ is {\em not} compactly supported.  In \secref{discussion} we discuss our results and some possible implications for statistical modeling.

\section{The case of compact support}
\label{sec:compact}

In this whole section we assume that $\psi$ has compact support.  Define the supremum of the support as follows 
\beq\label{support}
u_* = \sup\Big\{u : \int_{u - \eps}^u \psi(u) {\rm d} u > 0 \text{ for all } \eps > 0\Big\}.
\eeq
Note that $u_* < \infty$ by assumption and that the support of $\psi$ is included in $[0,u_*]$.
If $\psi$ is continuous (which the reader can assume without much loss of generality), then the following is an equivalent definition $u_* = \sup\{u: \psi(u) > 0\}$.
The emblematic example is that of the uniform distribution on the unit ball, in which case $\psi(u) = \bbI\{u \le 1\}$ and $u_* = 1$.  This distribution is well-known to concentrate near the boundary of its support (the unit sphere).  Our results below extend this to other distributions with compact support.

\subsection{Convergence in probability}
We start by establishing a convergence in probability.

\begin{thm} \label{thm:prob-compact}
In the setting considered here, $U_d \to u_*$ in probability as $d \to \infty$.
\end{thm}

\begin{proof}
Assume $u_* = 1$ without loss of generality.  
Then
\begin{align}
\P(U_d < 1-\eps)
= c_d \int_0^{1-\eps} u^d \psi(u) {\rm d}u
\le c_d (1-\eps)^d \int_0^1 \psi(u) {\rm d}u,
\end{align}
while
\beq
\P(U_d \ge 1-\eps) 
\ge \P(U_d \ge 1-\eps/2)
= c_d \int_{1-\eps/2}^1 u^d \psi(u) {\rm d}u
\ge c_d (1-\eps/2)^d \int_{1-\eps/2}^1 \psi(u) {\rm d}u.
\eeq
Note that the last integral is strictly positive for all $\eps > 0$ by definition of $u_*$ in \eqref{support} (recall that we assumed that $u_* = 1$).
Hence
\beq
\frac{\P(U_d < 1-\eps)}{\P(U_d \ge 1-\eps)} \le \frac{(1-\eps)^d \int_0^1 \psi(u) {\rm d}u}{(1-\eps/2)^d \int_{1-\eps/2}^1 \psi(u) {\rm d}u} \to 0, \quad d \to \infty,
\eeq
when $\eps \in (0,1)$ is fixed.  Since $\P(U_d < 1-\eps) + \P(U_d \ge 1-\eps) = 1$, we proved that $\P(U_d < 1-\eps) \to 0$ for all $\eps > 0$.  This, coupled with the fact that $\P(U_d \le 1) = 1$, proves that $U_d \to 1$ in probability as $d \to \infty$.
\end{proof}

\subsection{Convergence in distribution}
Beyond a convergence in probability, we can establish a convergence in distribution.  The limiting distribution happens to depend on the behavior of $\psi$ in the neighborhood of $u_*$.  We only cover the case where $\psi$ behaves as a power function near $u_*$.

\begin{thm} \label{thm:weak-compact}
In the setting considered here, assume in addition that $\psi$ is bounded and that $\psi(u) \sim a (u_*-u)^b$ as $u \nearrow u_*$ for some $a > 0$ and $b > -1$.  Then $d (u_* - U_d)$ converges weakly to the Gamma distribution with shape parameter $b+1$ and rate $1/u_*$.
\end{thm}

\begin{proof}
Assume without loss of generality that $u_* = 1$.
We first control the behavior of $c_d$ as $d \to \infty$.
Fix $\eps \in (0,1)$.
On the one hand, by the assumptions on $\psi$ and Dominated 
Convergence, we have
\beq
\int_{1-\eps}^1 u^d \psi(u) {\rm d}u 
\sim \int_{1-\eps}^1 u^d a (1-u)^b {\rm d}u
\sim a B(d+1, b+1), \quad d \to \infty,
\eeq
where $B$ is the Beta function.
On the other hand, by \thmref{prob-compact}, 
\beq
c_d \int_{1-\eps}^1 u^d \psi(u) {\rm d}u \sim 1, \quad d \to \infty.
\eeq
Together, this proves that
\beq
1/c_d \sim a B(d+1, b+1) \sim a \Gamma(b+1) d^{-(b+1)}, \quad d \to \infty,
\eeq
where $\Gamma$ is the Gamma function.

We now consider the case where $\eps = \eps_d \to 0$ as $d \to \infty$.  More precisely, we fix $t > 0$ and set $\eps_d = t/d$.  By Dominated Convergence again, applied twice, and a change of variables, as $d \to \infty$, we have
\begin{align}
\P(U_d > 1-t/d) 
&= c_d \int_{1-t/d}^1 u^d \psi(u) {\rm d}u \\
&\sim c_d \int_{1-t/d}^1 u^d a (1-u)^b {\rm d}u \\
&= c_d a d^{-(b+1)} \int_0^t (1 - v/d)^d v^b {\rm d}v \\
&\sim \frac1{\Gamma(b+1)} \int_0^t e^{-v} v^b {\rm d}v.
\end{align}
Recognizing the distribution function of the Gamma distribution with shape parameter $b+1$ and rate 1, the proof is complete. 
\end{proof}

\section{The case of non-compact support}
\label{sec:non-compact}

We now assume that $\psi$ has non-compact support, which is equivalent to $u_* = \infty$ in \eqref{support}.  
We note that here the emblematic example is that of the standard normal distribution, which is known to concentrate near the sphere of radius $\sqrt{d}$, meaning $U_d/\sqrt{d} \to 1$.  In fact, $U_d^2$ has the chi-squared distribution with $d$ degrees of freedom, and in particular, $\sqrt{2} (U_d - \sqrt{d})$ is asymptotically standard normal in the limit $d \to \infty$.  Our results below extend this phenomena to other radial distributions.  

While we were able to handle the case of compact support, which we treated in \secref{compact}, with very natural assumptions, the case of non-compact support appears more challenging and our working assumptions are more complicated.  This is despite the fact that we favored simplicity over generality.  Nevertheless, our working assumptions include interesting (and natural) examples.

\subsection{Convergence in probability}
We start by establishing a convergence in probability.

We start by making the following assumptions.
We assume there is $u_\ddag$ such that, for $u \ge u_\ddag$, $\Lambda(u) := - \log \psi(u)$ is differentiable and $L(u) := u \Lambda'(u)$ is increasing.  In addition, we assume that $M(u) := L(u)/\log(u) \to \infty$ as $u \to \infty$ and 
\beq\label{M}
\limsup_{u \to \infty} \frac{M((1-\eps)u)}{M(u)} \le 1, \quad
\liminf_{u \to \infty} \frac{M((1+\eps)u)}{M(u)} \ge 1, \quad
\forall \eps \in (0,1).
\eeq

\begin{thm} \label{thm:prob}
In the setting considered here, $U_d/u_d \to 1$ in probability as $d \to \infty$, where $u_d := L^{-1}(d)$.
\end{thm}

\begin{exa} \label{exa:running}
Consider the case where $\Lambda(u) = c \log(u+a)^\alpha (u+b)^\beta$, where $a > 0$, $b \ge 0$, $c > 0$, $\alpha \in \bbR$ and $\beta > 0$.  Surely, this defines a bonafide shape function $\psi$ in the sense of \eqref{psi}.  It can be shown that $\psi$ defined as such satisfies the conditions of \thmref{prob}, with 
\beq
u_d \sim c^{-1/\beta} \beta^{(\alpha-1)/\beta} (\log d)^{-\alpha/\beta} d^{1/\beta}, \quad d \to \infty.
\eeq
\end{exa}

\begin{proof}
The function $u \mapsto u^d \psi(u)$ is increasing on $[0,u_d)$ and decreasing on $(u_d, \infty)$.  Indeed, $\log(u^d \psi(u)) = d \log u - \Lambda(u)$ has derivative $\frac1u (d - L(u))$, which is positive for $u < u_d$, zero at $u = u_d$, and negative at $u > u_d$, by our assumptions and the definition of $u_d$.  Note that, necessarily, $u_d \to \infty$ as $d \to \infty$.

We have
\beq
\P(U_d \le v) = c_d \int_0^v u^d \psi(u) {\rm d}u.
\eeq
Fix $\eps \in (0,1)$.  

\noindent {\em Left tail.}
Using the fact that $u^d \psi(u) \le u_0^d \psi(u_0)$ for any $u \le u_0 \le u_d$, we have
\begin{align}
\frac1{c_d} \P(U_d \le (1-\eps)u_d) 
&= \int_0^{(1-\eps)u_d} u^d \psi(u) {\rm d}u \\
&\le ((1-\eps)u_d)^{d+1} \psi((1-\eps)u_d),
\end{align}
and we also have
\begin{align}
\frac1{c_d} \P(U_d \ge (1-\eps)u_d) 
&= \int_{(1-\eps)u_d}^\infty u^d \psi(u) {\rm d}u \\
&\ge \int_{(1-\eps/2)u_d}^{u_d} u^d \psi(u) {\rm d}u \\
\label{pleft_deno}
&\ge (\eps/2) u_d ((1-\eps/2)u_d)^d \psi((1-\eps/2)u_d).
\end{align}
Taking the ratio, we obtain
\begin{align}
\frac{\P(U_d \le (1-\eps)u_d)}{\P(U_d \ge (1-\eps)u_d)} 
&\le \frac{((1-\eps)u_d)^{d+1} \psi((1-\eps)u_d)}{(\eps/2) u_d ((1-\eps/2)u_d)^d \psi((1-\eps/2)u_d)} \label{lower} \\
&\le \frac{1-\eps}{\eps/2} \left(\frac{1-\eps}{1-\eps/2}\right)^d \frac{\psi((1-\eps)u_d)}{\psi((1-\eps/2)u_d)}. \label{left0}
\end{align}
Applying the logarithm, and ignoring the constant factor, we further get
\begin{align}
& d \log \left(\frac{1-\eps}{1-\eps/2}\right) - \Lambda((1-\eps)u_d) + \Lambda((1-\eps/2)u_d) \\
& = - d \int_{(1-\eps)u_d}^{(1-\eps/2)u_d} \frac1u {\rm d}u + \int_{(1-\eps)u_d}^{(1-\eps/2)u_d} \Lambda'(u) {\rm d}u \\
&= - \int_{(1-\eps)u_d}^{(1-\eps/2)u_d} (d - L(u)) \frac1u {\rm d}u \\
\label{log_left}
&\le - (d - L((1-\eps/2)u_d)) \log \left(\frac{1-\eps/2}{1-\eps}\right),
\end{align}
where we used the monotonicity of $L$ in the last line.
Therefore, to show that the fraction in \eqref{lower} converges to 0, it suffices to show that $d - L((1-\eps/2)u_d) \to \infty$.  The limit is as $d \to \infty$ while $\eps$ remains fixed.  Using the fact that $L(u_d) = M(u_d) \log u_d = d$, we have
\begin{align}
d - L((1-\eps/2)u_d)
&= d - M((1-\eps/2)u_d) \log((1-\eps/2)u_d) \\
&= d - M((1-\eps/2)u_d) \log(1-\eps/2) - M((1-\eps/2)u_d) \frac{d}{M(u_d)} \\
&= -M((1-\eps/2)u_d) \log(1-\eps/2) + d\left[1 - \frac{M((1-\eps/2)u_d)}{M(u_d)}\right].
\end{align}
In the last line, the first term tends to infinity because $M(u) \to \infty$ as $u\to \infty$ and $u_d \to \infty$, while the second terms is nonnegative in the limit because of \eqref{M}, so that the last expression tends to infinity.

We conclude that, for the left tail, 
\beq
\P(U_d \le (1-\eps)u_d) \to 0, \quad d \to \infty.
\eeq

\noindent {\em Right tail.}
Using the fact that $u^\ell \psi(u) \le u_0^\ell \psi(u_0)$ for any $u \ge u_0 \ge u_\ell$, where $u_\ell = L^{-1}(\ell)$ in congruence with our definition above, and assuming for now that $b := L((1+\eps)u_d) > d+1$, we have
\begin{align}
\frac1{c_d} \P(U_d \ge (1+\eps)u_d) 
&= \int_{(1+\eps)u_d}^\infty u^d \psi(u) {\rm d}u \\
&\le ((1+\eps)u_d)^b \psi((1+\eps)u_d) \int_{(1+\eps)u_d}^\infty u^{d - b} {\rm d}u \\
&= ((1+\eps)u_d)^b \psi((1+\eps)u_d) \frac{((1+\eps)u_d)^{d+1-b}}{b-d-1} \\
&=  \frac{((1+\eps)u_d)^{d+1}}{b-d-1} \psi((1+\eps)u_d),
\end{align}
and we also have
\begin{align}
\frac1{c_d} \P(U_d \le (1+\eps)u_d) 
&= \int_0^{(1+\eps)u_d} u^d \psi(u) {\rm d}u \\
&\ge \int_{u_d}^{(1+\eps/2)u_d} u^d \psi(u) {\rm d}u \\
&\ge (\eps/2) u_d ((1+\eps/2)u_d)^d \psi((1+\eps/2)u_d).
\end{align}
Taking the ratio, we obtain
\begin{align}
\frac{\P(U_d \ge (1+\eps)u_d)}{\P(U_d \le (1+\eps)u_d)} 
&\le \frac{\frac{((1+\eps)u_d)^{d+1}}{b-d-1} \psi((1+\eps)u_d)}{(\eps/2) u_d ((1+\eps/2)u_d)^d \psi((1+\eps/2)u_d)} \label{upper} \\
\label{right0}
&\le \frac{1+\eps}{\eps/2} \frac1{b-d-1} \left(\frac{1+\eps}{1+\eps/2}\right)^d \frac{\psi((1+\eps)u_d)}{\psi((1+\eps/2)u_d)}.
\end{align}

We pause to show that $b-d \to \infty$ eventually.  This is because, using the fact that $L(u_d) = M(u_d) \log u_d = d$, 
\begin{align}
b -d 
= L((1+\eps)u_d) -d 
&= M((1+\eps)u_d) \log ((1+\eps)u_d) -d \label{b} \\
&= M((1+\eps)u_d) \log (1+\eps) + \left[\frac{M((1+\eps)u_d)}{M(u_d)} -1\right] d.
\end{align}
In the last line, the first term tends to infinity because $M(u) \to \infty$ as $u\to \infty$ and $u_d \to \infty$, while the second terms is nonnegative in the limit because of \eqref{M}, so that the last expression tends to infinity.

Returning to \eqref{upper}, applying the logarithm, and ignoring the first two factors whose product is bounded by 1 eventually, we further get
\begin{align}
& d \log \left(\frac{1+\eps}{1+\eps/2}\right) - \Lambda((1+\eps)u_d) + \Lambda((1+\eps/2)u_d) \\
& = d \int_{(1+\eps/2)u_d}^{(1+\eps)u_d} \frac1u {\rm d}u - \int_{(1+\eps/2)u_d}^{(1+\eps)u_d} \Lambda'(u) {\rm d}u \\
& = - \int_{(1+\eps/2)u_d}^{(1+\eps)u_d} (L(u) -d) \frac1u {\rm d}u \\
\label{right_2nd}
&\le - (L((1+\eps/2)u_d) -d) \log \left(\frac{1+\eps}{1+\eps/2}\right),
\end{align}
where we used the monotonicity of $L$ in the last line.
Therefore, to show that the fraction in \eqref{upper} converges to 0, it suffices to show that $L((1+\eps/2)u_d) -d \to \infty$, and we already did this in \eqref{b}.

We conclude that, for the right tail, 
\beq
\P(U_d \ge (1+\eps)u_d) \to 0, \quad d \to \infty.
\eeq

\medskip
We can therefore conclude that $U_d/u_d \to 1$ in probability as $d \to \infty$.
\end{proof}

\subsection{Convergence in distribution}
\label{sec:weak}
We now turn to establishing a convergence in distribution.  Although we speculate that other cases may arise, we give (additional) sufficient conditions for a Gaussian limit.

We make the following additional assumptions.  We assume that $L$ is differentiable with $\nu_d := u_d L'(u_d) \to \infty$ and that there is $\omega_d \to \infty$ such that 
\beq\label{taylor}
\big|L(u_d) - L((1-\eps) u_d) - \eps u_d L'(u_d)\big| \le |\eps| \nu_d/\omega_d, \quad \text{whenever } \eps^2 \le \omega_d/\nu_d.  
\eeq
Note that \eqref{taylor} is a form of first-order Taylor expansion around $u_d$.

The following refines \thmref{prob}.
\begin{prp} \label{prp:prob}
In the setting considered here, 
\beq\label{eps_d}
\P(1 - \eps_d \le U/u_d \le 1 + \eps_d) \to 1, \quad \text{whenever } \eps_d \gg 1/\sqrt{\nu_d}.
\eeq
\end{prp}

\begin{proof}
By mononicity in $\eps_d > 0$, it is enough to show that when $ \eps_d^2 \le \omega_d/\nu_d$.  Then to prove \eqref{eps_d}, as before, it suffices to show that
\beq
\text{(left tail) } \frac{\P(U_d \le (1-\eps_d)u_d)}{\P(U_d \ge (1-\eps_d)u_d)} \to 0 \quad \text{and} \quad 
\text{(right tail) } \frac{\P(U_d \ge (1+\eps_d)u_d)}{\P(U_d \le (1+\eps_d)u_d)} \to 0.
\eeq

\noindent {\em Left tail.} As before, we can show that $b := L((1-\eps_d)u_d)$ satisfies $b -d \to -\infty$, so that we may assume that $b < d-1$.  Then using the fact that $u^\ell \psi(u) \le u_0^\ell \psi(u_0)$ for any $u \le u_0 \le u_\ell$, where $u_\ell = L^{-1}(\ell)$, we have 
\begin{align}
\frac1{c_d} \P(U_d \le (1-\eps_d)u_d) 
&= \int_0^{(1-\eps_d)u_d} u^d \psi(u) {\rm d}u \\
&\le ((1-\eps_d)u_d)^b \psi((1-\eps_d)) \int_0^{(1-\eps_d)u_d} u^{d-b}{\rm d} u\\
\label{pleft_num}
&= \frac{((1-\eps_d)u_d)^{d+1}}{d-b+1} \psi((1-\eps_d)u_d).
\end{align}
Taking the ratio of \eqref{pleft_num} to \eqref{pleft_deno} (but replacing $\eps$ by $\eps_d$), we obtain
\beq
\frac{\P(U_d \le (1-\eps_d)u_d)}{\P(U_d \ge (1-\eps_d)u_d)} \leq \frac1{d-b+1}\frac{1-\eps_d}{\eps_d/2} \left(\frac{1-\eps_d}{1-\eps_d/2}\right)^d \frac{\psi((1-\eps_d)u_d)}{\psi((1-\eps_d/2)u_d)}.
\eeq
As in \eqref{left0} and \eqref{log_left}, we apply a logarithm, and obtain the upper bound
\beq
\label{left}
\log \frac{1-\eps_d}{\frac{\eps_d}{2} (d-b+1)} - [d-L((1-\eps_d/2)u_d)] \log \frac{1-\eps_d/2}{1-\eps_d}.
\eeq
By \eqref{taylor}, which is applicable by our assumption $\eps_d^2 \le \omega_d/\nu_d$, we have
\beq
d-b = L(u_d) - L((1-\eps_d)u_d) = \eps_d \nu_d \pm |\eps_d| \nu_d/\omega_d \sim \eps_d \nu_d, \quad d \to \infty.
\eeq
Similarly,
\beq
d-L((1-\eps_d/2)u_d) \sim \tfrac12 \eps_d \nu_d, \quad d \to \infty.
\eeq
Also, note that $\eps_d^2 \nu_d \to \infty$ by assumption.
Hence, the first term in \eqref{left} is $\sim -\log(\eps_d^2 \nu_d)$ while the second term (including sign) is $\sim -\frac14 \eps_d^2 \nu_d$, so that the sum tends to $-\infty$.

\bigskip 
\noindent {\em Right tail.} 
The treatment of the right tail is analogous, starting with \eqref{right0} instead of \eqref{left0}.  Details are omitted.  
\end{proof}

In the following we examine the behavior of the normalizing constant $c_d$ as $d \to \infty$.
\begin{prp}\label{prp:constant}
In the setting considered here,
\beq
\label{constant}
\frac1{c_d} \sim \sqrt{\frac{2 \pi}{\nu_d}} u_d^{d+1} \psi(u_d), \quad d \to \infty.
\eeq
\end{prp}

\begin{proof}
Let $\eps_d$ be such that $1/\nu_d \ll \eps_d^2 \ll \omega_d/\nu_d$.  Applying \prpref{prob} and then performing a change of variables, we get
\begin{align} 
\frac1{c_d} 
= \int_0^\infty u^d \psi(u) {\rm d}u
&\sim \int_{(1-\eps_d) u_d}^{(1+\eps_d) u_d} u^d \psi(u) {\rm d}u \\
&\sim \int_{-\eps_d}^{\eps_d} [(1+t) u_d]^d \psi[(1+t) u_d] u_d {\rm d}t \\
&= u_d^{d+1} \psi(u_d) \int_{-\eps_d}^{\eps_d} (1+t)^d \frac{\psi[(1+t) u_d]}{\psi(u_d)}{\rm d}t. \label{eq1}
\end{align}
As before,
\begin{align} 
\log \Big\{(1+t)^d \frac{\psi[(1+t) u_d]}{\psi(u_d)}\Big\}
&= d \log(1+t) - \Lambda[(1+t)u_d] + \Lambda(u_d) \\
&= d \int_0^t \frac{s}{1+s} {\rm d}s - \int_0^t u_d \Lambda'[(1+s) u_d] {\rm d}s \\
&= \int_0^t \frac1{1+s} \Big\{d - L[(1+s) u_d]\Big\} {\rm d}s.
\end{align}
Noting that $|s| \le \eps_d$, and using \eqref{taylor}, we get
\beq
- s \nu_d - s \nu_d/\omega_d \le d - L[(1+s) u_d] 
= L(u_d) - L[(1+s) u_d]
\leq - s \nu_d + s \nu_d/\omega_d.
\eeq
Hence,
\begin{align}
- \nu_d (1 + 1/\omega_d) \int_0^t \frac{s}{1+s} {\rm d}s \le \int_0^t \frac1{1+s} \Big\{d - L[(1+s) u_d]\Big\} {\rm d}s &\le - \nu_d (1 - 1/\omega_d) \int_0^t \frac{s}{1+s} {\rm d}s,
\end{align}
with
\beq
\int_0^t \frac{s}{1+s} {\rm d}s
= \tfrac12 t^2 + O(t^3)
= \tfrac12 t^2 + O(\eps_d^3),
\eeq
since $|t| \le \eps_d$, so that
\beq
\int_0^t \frac1{1+s} \Big\{d - L[(1+s) u_d]\Big\} {\rm d}s
= - \tfrac12 t^2 \nu_d + O(\eps_d + 1/\omega_d) \eps_d^2 \nu_d,
\eeq
where the big-O is uniform in $t \in [-\eps_d, \eps_d]$.
We already took $\eps_d$ such that $(1/\omega_d) \eps_d^2 \nu_d \to 0$, and it is compatible to choose $\eps_d$ such that, in addition, $\eps_d^3 \nu_d \to 0$.  When we do so, the remainder term above is $o(1)$, and in particular, 
\beq
\int_0^t \frac1{1+s} \Big\{d - L[(1+s) u_d]\Big\} {\rm d}s = - \tfrac12 t^2 \nu_d + o(1),
\eeq
where the $o(1)$ term is uniform in $t \in [-\eps_d, \eps_d]$.
With such a choice of $\eps_d$, we continue our derivations above
\begin{align}
\int_{-\eps_d}^{\eps_d} (1+t)^d \frac{\psi[(1+t) u_d]}{\psi(u_d)}{\rm d}t
&= \int_{-\eps_d}^{\eps_d} \exp\Big\{- \tfrac12 t^2 \nu_d + o(1) \Big\} {\rm d}t \\
&\sim \frac1{\sqrt{\nu_d}} \int_{-\eps_d \sqrt{\nu_d}}^{\eps_d \sqrt{\nu_d}} \exp\Big\{- \tfrac12 s^2 \Big\} {\rm d}s 
\sim \frac1{\sqrt{\nu_d}} \sqrt{2 \pi}, \label{2pi}
\end{align}
since $\eps_d \sqrt{\nu_d} \to \infty$ by assumption.
\end{proof}

We are finally equipped to establish a convergence in distribution for $U_d$.

\begin{thm} \label{thm:weak}
In the setting considered here, $\sqrt{\nu_d} (U_d/u_d - 1)$ converges weakly to the standard normal distribution as $d \to \infty$.
\end{thm}

\setcounter{exa}{0}
\begin{exa}[Continued]
It can be checked that the same example of shape function $\psi$ satisfies the conditions assumed here, with $u L'(u) \sim c \beta^2 (\log u)^\alpha u^\beta$ 
as $u \to \infty$, so that
\beq
\nu_d = u_d L'(u_d) \sim \beta d, \quad d \to \infty.
\eeq
\end{exa}

\begin{proof}
Fix $r \in \bbR$ and let $\eps_d$ be as before.  
As in \eqref{eq1}, we get
\begin{align}
\P(\sqrt{\nu_d} (U_d/u_d - 1) \le r)
&= \P(U_d \le (1+r/\sqrt{\nu_d}) u_d) \\
&\sim \P((1-\eps_d) u_d \le U_d \le (1+r/\sqrt{\nu_d}) u_d) \\
&= c_d u_d^{d+1} \psi(u_d) \int_{-\eps_d}^{r/\sqrt{\nu_d}} (1+t)^d \frac{\psi[(1+t) u_d]}{\psi(u_d)}{\rm d}t.
\end{align}
Again, as before,
\begin{align}
\int_{-\eps_d}^{r/\sqrt{\nu_d}} (1+t)^d \frac{\psi[(1+t) u_d]}{\psi(u_d)}{\rm d}t 
&\sim \int_{-\eps_d}^{r/\sqrt{\nu_d}} \exp\Big\{- \tfrac12 t^2 \nu_d \Big\} {\rm d}t \\
&= \frac1{\sqrt{\nu_d}} \int_{-\eps_d \sqrt{\nu_d}}^r \exp\Big\{- \tfrac12 s^2 \Big\} {\rm d}s 
\sim \frac1{\sqrt{\nu_d}} \sqrt{2 \pi} \Phi(r),
\end{align}
where $\Phi$ is the standard normal distribution function.
We then combine this with \prpref{constant}.
\end{proof}

\section{Discussion}
\label{sec:discussion}

While there is relatively little related work, a detailed comparison with \citep{sherlock2012class} is in order.\footnote{Note that we only became aware of that work after we posted our paper on {\tt arxiv.org}, which explains some of the nontrivial overlap.}  \citeauthor{sherlock2012class} focus on the non-compact case --- corresponding to \secref{non-compact} here.  They derive the same result as our \thmref{prob} under different conditions.  They require that $\eta(u) := \Lambda(\exp(u))$ is twice differentiable with $\eta''(u) \to \infty$ as $u \to \infty$, while our condition is a bit weaker than requiring that $\eta$ is once differentiable with $\eta'(u)/u \to \infty$ and increasing.  
Note that their condition is equivalent to requiring that $L$ is differentiable with $u L'(u) \to \infty$, a condition that arises in \secref{weak}.  
\citeauthor{sherlock2012class} do not establish weak convergence, however, but they obtain other results.  In particular, they establish concentration for a marginal of $X_d$ and also for the maximum of the marginals.  In addition, they extend their results to the case of elliptical distributions under conditions on the eigenvalues of the scaling matrix.

\bigskip
{\em What are possible consequences for statistical modeling?}  
Because of the weak convergence of the sort established here, the behavior of $U_d$ is asymptotically characterized solely by a few parameters of the underlying distribution.  For example, if the conditions of \thmref{weak-compact} are fulfilled, then the distribution of $U_d$ in the large-dimension limit ($d \to \infty$) only depends on $u_*$ (irrelevant in practice because scale is typically estimated) and the behavior of $\psi$ near $u_*$.  In particular, whether $\psi(u) = \bbI\{u \le 1\}$ or $\psi(u) = (2 - u) \bbI\{u \le 1\}$, in both cases, $d (1 - U_d)$ converges weakly to the exponential distribution with rate 1.  This means that, in order to even distinguish two such distributions with nontrivial accuracy, we require a sample of size that increases with $d$.  (We did not attempt to quantify this further, although this is possible by framing the problem as a hypothesis testing problem.)  
A similar phenomenon arises with certain distributions with non-compact support, based on our \thmref{weak}.  Thus, if the sample size is small relative to the dimension, very different shape functions (meaning, different $\psi$'s) could yield indistinguishable models.

The flip side of this is a form of universality of the Gaussian distribution, in particular, in context such as Linear Discriminant Analysis (classification) or Gaussian Mixture Modeling (clustering).  Surely, both LDA and GMM have computational advantages over other methods (the latter using the EM algorithm, for example).  Beyond this important computational aspect, our results indicate that when the sample is small relative to the dimension, fitting a Gaussian model may be, in fact, indistinguishable from fitting another model base on a shape function having similar characteristics as the standard normal distribution that dictate the asymptotic behavior of $U_d$.


{\small
\bibliographystyle{chicago}
\bibliography{ref}

\begin{thebibliography}{}

\bibitem[\protect\citeauthoryear{Anderson}{Anderson}{2003}]{MR1990662}
Anderson, T.~W. (2003).
\newblock {\em An introduction to multivariate statistical analysis\/} (Third
  ed.).
\newblock Wiley Series in Probability and Statistics. Wiley-Interscience [John
  Wiley \& Sons], Hoboken, NJ.

\bibitem[\protect\citeauthoryear{Bhattacharyya and Bickel}{Bhattacharyya and
  Bickel}{2015}]{bhattacharyyaadaptive}
Bhattacharyya, S. and P.~J. Bickel (2015).
\newblock Adaptive estimation in elliptical distributions with extensions to
  high dimensions.
\newblock Technical report, University of California, Berkeley.

\bibitem[\protect\citeauthoryear{Bickel, Klaassen, Ritov, Wellner,
  et~al.}{Bickel et~al.}{1998}]{bickel1998efficient}
Bickel, P.~J., C.~A. Klaassen, Y.~Ritov, J.~A. Wellner, et~al. (1998).
\newblock Efficient and adaptive estimation for semiparametric models.

\bibitem[\protect\citeauthoryear{Boucheron, Lugosi, and Massart}{Boucheron
  et~al.}{2013}]{boucheron2013concentration}
Boucheron, S., G.~Lugosi, and P.~Massart (2013).
\newblock {\em Concentration inequalities: A nonasymptotic theory of
  independence}.
\newblock Oxford university press.

\bibitem[\protect\citeauthoryear{Chang and Walther}{Chang and
  Walther}{2007}]{chang2007clustering}
Chang, G.~T. and G.~Walther (2007).
\newblock Clustering with mixtures of log-concave distributions.
\newblock {\em Computational Statistics \& Data Analysis\/}~{\em 51\/}(12),
  6242--6251.

\bibitem[\protect\citeauthoryear{Diaconis and Freedman}{Diaconis and
  Freedman}{1984}]{diaconis1984asymptotics}
Diaconis, P. and D.~Freedman (1984).
\newblock Asymptotics of graphical projection pursuit.
\newblock {\em The Annals of Statistics\/}~{\em 12\/}(3), 793--815.

\bibitem[\protect\citeauthoryear{Diaconis and Freedman}{Diaconis and
  Freedman}{1987}]{diaconis1987dozen}
Diaconis, P. and D.~Freedman (1987).
\newblock A dozen de finetti-style results in search of a theory.
\newblock {\em Annales de l'Institut Henri Poincar\'e, Probabilit\'es et
  Statistiques\/}~{\em 23\/}(S2), 397--423.

\bibitem[\protect\citeauthoryear{Ledoux}{Ledoux}{2005}]{ledoux2005concentration}
Ledoux, M. (2005).
\newblock {\em The concentration of measure phenomenon}.
\newblock Number~89. American Mathematical Soc.

\bibitem[\protect\citeauthoryear{Lewis}{Lewis}{1998}]{lewis1998naive}
Lewis, D.~D. (1998).
\newblock Naive (bayes) at forty: The independence assumption in information
  retrieval.
\newblock In {\em European conference on machine learning}, pp.\  4--15.
  Springer.

\bibitem[\protect\citeauthoryear{Sherlock and Elton}{Sherlock and
  Elton}{2012}]{sherlock2012class}
Sherlock, C. and D.~Elton (2012).
\newblock A class of spherical and elliptical distributions with gaussian-like
  limit properties.
\newblock {\em Journal of Probability and Statistics\/}~{\em 2012}.

\end{thebibliography}
}
\end{document}